\documentclass[final,1p,number,11pt]{elsarticle}
\usepackage{amsmath,amssymb,amsthm}

\newtheorem{theorem}{Theorem}[section]
\newtheorem{lemma}[theorem]{Lemma}
\newtheorem{corollary}[theorem]{Corollary}

\theoremstyle{definition}

\newtheorem{remark}[theorem]{Remark}

\DeclareMathOperator{\RE}{Re} 
\newcommand{\UCV}{\mathcal{UCV}}
\newcommand{\TS}{\mathcal{TS}^*}
\newcommand{\TC}{\mathcal{TC}}
\begin{document}

\begin{frontmatter}
\title{Coefficient Inequalities for Starlikeness and Convexity  }

\author[rma]{Rosihan M. Ali}
\address[rma]{School of Mathematical Sciences, Universiti Sains
Malaysia, 11800 Penang, Malaysia} \ead{rosihan@cs.usm.my}

\author[rma]{Mahnaz M. Nargesi}
 \ead{moradinargesik@yahoo.com}

\author[rma,vr]{V. Ravichandran}
\address[vr]{Department of Mathematics,
University of Delhi, Delhi 110 007, India}
\ead{vravi@maths.du.ac.in}

\begin{abstract}
For an analytic function $f(z)=z+\sum_{n=2}^\infty a_n z^n$
satisfying the inequality $\sum_{n=2}^\infty n(n-1)|a_n|\leq \beta$,
the range of $\beta$ is determined  so that the function $f$ is
either starlike or convex of order $\alpha$. Several related
problems are also investigated. Applications of these results to
Gaussian hypergeometric functions are also provided.
\end{abstract}


\end{frontmatter}

\section{Introduction}\noindent
Let $\mathcal{A}$ be the class of analytic functions  in $\mathbb{D}
=\{z\in\mathbb{C}: |z| <1\}$, normalized by $f(0)=0$ and $f'(0)=1$.
A function $f \in \mathcal{A}$ has Taylor's series expansion of the
form \begin{equation}\label{f(z)} f(z)=z+\sum_{n=2}^\infty a_n
z^n.\end{equation} Let $\mathcal{S}$ be the subclass of
$\mathcal{A}$ consisting of univalent functions. For $ 0\leq
\alpha<1$, let $\mathcal{S}^* (\alpha)$ and $\mathcal{C}(\alpha)$ be
subclasses of $\mathcal{S}$ consisting of starlike functions of
order $\alpha$ and convex functions of order $\alpha$, respectively
defined analytically by the following equalities:
\[\mathcal{S}^*(\alpha):=\left\{ f\in\mathcal{S}: \RE \left( \frac{zf'(z)}{f(z)} \right)
> \alpha\right\}, \
\text{ and } \ \mathcal{C}(\alpha):=\left\{f\in\mathcal{S}: \RE
\left(1+ \frac{zf''(z)}{f'(z)} \right) >\alpha \right\}.
\] The classes $\mathcal{S}^*:=\mathcal{S}^*(0)$ and
$\mathcal{C}:=\mathcal{C}(0)$ are the familiar classes of starlike
and convex functions respectively. Closely related are the following
classes of functions:
\[\mathcal{S}^*_\alpha  :=\left\{ f\in\mathcal{S}:   \left| \frac{zf'(z)}{f(z)} -1
\right| < 1- \alpha\right\}, \quad \text{ and } \quad
 \mathcal{C}_\alpha :=\left\{ f\in\mathcal{S}:  \left|\frac{zf''(z)}{f'(z)}
\right| < 1-\alpha\right\} .
\] Note that $ \mathcal{S}^*_\alpha \subseteq\mathcal{S}^*(\alpha)$
and $\mathcal{C}_\alpha\subseteq \mathcal{C}(\alpha)$. For
$\beta<1$, $\alpha\in \mathbb{R}$, a function  $f\in\mathcal{A}$
belongs to the class $\mathcal{R}(\alpha, \beta)$ if  the function
$f$ satisfies the following inequality
\begin{equation}\label{R} \RE
\left(\frac{zf'(z)}{f(z)}\left(\alpha\frac{zf''(z)}{f'(z)}+1
\right)\right)>\beta.\end{equation} Clearly,
$\mathcal{R}(0,\beta)=\mathcal{S}^*(\beta)$. For $\beta\geq
-\alpha/2$,  Li and Owa \cite{owa2} proved that
$\mathcal{R}(\alpha,\beta)\subset \mathcal{S}^*$.

A function $f\in \mathcal{S}$ is \emph{$k$-uniformly convex,}
$(0\leq k<\infty)$, if $f$ maps every circular arc $\gamma$
contained in $\mathbb{D}$ with center $\zeta$, $|\zeta|\leq k$, onto
a convex arc. The class of all $k$-uniformly convex functions is
denoted by $k-\UCV$. Goodman \cite{goodman} introduced the class
$\UCV:= 1-\UCV$ while the class $k-\UCV$ was introduced by    Kanas\
and Wisniowska \cite{kanas}. They  \cite[Theorem 2.2, p.\
329]{kanas} (see \cite{ravi} for details) have shown that $f\in
k-\UCV$ if and only if the function $f$ satisfies the following
inequality:
\[ k\left|\frac{zf''(z)}{f'(z)} \right| <
\RE \left(1+ \frac{zf''(z)}{f'(z)}\right). \] By making use of this
result, the following sufficient condition for a function to be
$k$-uniformly convex was proved in \cite{kanas}:

\begin{theorem}[{\cite[Theorem 3.3, p.\ 334]{kanas}}]\label{T1}
If the function $f(z)=z+\sum_{n=2}^\infty a_n z^n$ satisfies the
inequality  $\sum_{n=2}^\infty n(n-1)|a_n| \leq 1/(k+2)$, $(0\leq
k<\infty)$, then $f\in k-\UCV$. The bound $1/(k+2)$ cannot be
replaced by a larger number.
\end{theorem}

Goodman \cite[Theorem 6]{goodman}  proved the result in the case
$k=1$ for functions to be uniformly convex. Theorem~\ref{T1} in the
special case $k=0$ shows that the corresponding constant is 1/2 for
functions $f$ to be convex. A function $f\in \mathcal{A}$ is
\emph{parabolic starlike of order $\alpha$} \cite{ali1} if
\[\left| \frac{zf'(z)}{f(z)} -1 \right| < 1-2\alpha +
\RE\left(\frac{zf'(z)}{f(z)} \right).\]

\begin{theorem}[Ali {\cite[Theorem 3.1, p.\ 564]{ali1}}]\label{T2}
If the function  $f(z)=z+\sum_{n=2}^\infty a_n z^n$ satisfies the
inequality $\sum_{n=2}^{\infty}(n-1)|a_n|\leq
(1-\alpha)/(2-\alpha)$, then   the function $f$ is parabolic
starlike of order $\alpha$. The bound $(1-\alpha)/(2-\alpha)$ cannot
be replaced by a larger number.
\end{theorem}

Motivated by Theorems~\ref{T1} and \ref{T2}, the range of $\beta$ is
determined for the analytic function $f(z)=z+\sum_{n=2}^\infty a_n
z^n$ satisfying the inequality $\sum_{n=2}^\infty n(n-1)|a_n|\leq
\beta$ to be  either starlike or convex of order $\alpha$. Similar
problems were investigated for functions satisfying certain other
coefficient inequality. The reverse implications are investigated
for analytic functions with negative coefficients. Finally,
applications of these results to hypergeometric functions are also
provided.

The following theorem that gives necessary and  sufficient
conditions for functions to belong to certain subclasses of starlike
and  convex functions will be need in the sequel.

\begin{theorem}[{\cite[Theorem 2, p.\ 961]{merkes}}, and
{\cite[Theorem 1 and Corollary, p.\ 110]{silver}}] \label{suffi}
\begin{enumerate}\item[]
\item \label{suffi1a}  If the function $f(z)=z+\sum_{n=2}^\infty a_n
z^n$   satisfies the inequality
\begin{equation} \label{suffi1}
\sum_{n=2}^{\infty}(n-\alpha)|a_n|\leq 1-\alpha, \end{equation} then
$f\in \mathcal{S}^*_\alpha$. If $a_n\leq 0$, then the condition
\eqref{suffi1} is a necessary condition for $f\in
\mathcal{S}^*(\alpha)$.

\item    \label{suffi1b} Similarly, if the function $f$   satisfies the inequality
\begin{equation} \label{suffi2} \sum_{n=2}^{\infty}n(n-\alpha)|a_n|\leq
1-\alpha,\end{equation}  then $f\in \mathcal{C}_\alpha$. If $a_n\leq
0$, then the condition \eqref{suffi2} is a necessary condition for
$f\in \mathcal{C}(\alpha)$.
\end{enumerate}
\end{theorem}

The  necessary and  sufficient conditions in
Theorem~\ref{suffi}\eqref{suffi1a}  was proved by Merkes, Robertson,
and Scott \cite[Theorem 2, p.\ 961]{merkes} in 1962 and proved
independently by Silverman \cite[Theorem 1, p.\ 110]{silver} later
in 1975. The corresponding result  in
Theorem~\ref{suffi}\eqref{suffi1b} follows by the Alexandar's result
and it was proved in \cite[Corollary, p.\ 110]{silver}.

\section{Sufficient conditions for starlikeness and
convexity}\noindent The following theorem provides sufficient
coefficient inequality for functions to be in the classes
$\mathcal{C}_\alpha$ or $\mathcal{S}^*_\alpha$.

\begin{theorem}\label{th1}
Let $\alpha\in[0,1)$. If the function   $f\in \mathcal{A}$ given by
\eqref{f(z)} satisfies the inequality
\begin{equation}\label{th2.1eq1}
 \sum_{n=2}^{\infty}n(n-1)|a_n|\leq \beta<1,
\end{equation}
then the following holds.
\begin{enumerate}
  \item[(1)] The function $f$ belongs to the class $\mathcal{C}_\alpha$ for $\beta\leq (1-\alpha)/(2-\alpha)$, and
the bound $(1-\alpha)/(2-\alpha)$ cannot be replaced by a larger
number.

  \item[(2)]  The function $f$ belongs to the class $ \mathcal{S}^*_\alpha$ for $\beta\leq
  2(1-\alpha)/(2-\alpha)$, and the bound $2(1-\alpha)/(2-\alpha)$ cannot be replaced by a
larger number.

\end{enumerate}
\end{theorem}

\begin{proof} (1) Let the function $f$ satisfy the inequality
\eqref{th2.1eq1} with $\beta\leq (1-\alpha)/(2-\alpha)$. Then, the
Equation~\eqref{th2.1eq1} together with \eqref{f(z)}   shows that
\[|f'(z)-1|\leq \sum_{n=2}^{\infty}n |a_n|\leq
\sum_{n=2}^{\infty}n(n-1)|a_n|<1,\] and so $\RE f'(z)>0$. This shows
that  $f\in \mathcal{S}$. Since the inequality
\begin{equation}\label{th2.1eq2}
 n-\alpha \leq (2-\alpha)(n-1)
\end{equation}  holds  for
$n\geq 2$, the inequality \eqref{th2.1eq1} leads to
\[
\sum_{n=2}^{\infty}n(n-\alpha)|a_n|
\leq(2-\alpha)\sum_{n=2}^{\infty}n(n-1)|a_n| \leq (2-\alpha) \beta
\leq  1-\alpha .
\] Thus, by Theorem~\ref{suffi}(2),  $f\in\mathcal{C}_\alpha$. The
function $f_0:\mathbb{D}\rightarrow \mathbb{C}$ defined by
\[f_0(z)=z-\frac{1}{2}\frac{1-\alpha}{2-\alpha}z^2\] satisfies the
hypothesis of Theorem~\ref{suffi} and therefore $f_0\in
\mathcal{C}_\alpha$. This function $f_0$ shows that the bound for
$\beta$ cannot be replaced by a larger number.

(2) Now, let the function $f$ satisfy the inequality
\eqref{th2.1eq1} with $\beta\leq 2(1-\alpha)/(2-\alpha)$. When
$n\geq 2$, the inequality  \eqref{th2.1eq2} holds and this leads to
\[ (n-\alpha) \leq \frac{n(n-\alpha)}{2}\leq \frac{(2-\alpha)n(n-1)}{2} \quad (n\geq 2)\] and hence
\[ \sum_{n=2}^{\infty}(n-\alpha)|a_n|
\leq\frac{(2-\alpha)}{2}\sum_{n=2}^{\infty}n(n-1)|a_n| \leq
(1-\alpha).\] By Theorem \ref{suffi}(1), $ f\in
\mathcal{S}^*_\alpha$.
 The function
\[f_0(z)=z- \frac{1-\alpha}{2-\alpha}z^2\in
\mathcal{S}^*_\alpha\] shows that the result is sharp.
\end{proof}

\begin{corollary}{\rm\cite[Theorem 3.3, p. 334]{kanas}} If $f\in \mathcal{A}$ given
by \eqref{f(z)} satisfies the inequality
\[\sum_{n=2}^{\infty}n(n-1)|a_n|\leq \frac{1}{k+2},\]
then $f\in  k-\UCV$. Further, the bound $1/(k+2)$ cannot be replaced
by a larger number.
\end{corollary}

\begin{proof} By Theorem~\ref{th1}(1), it follows that $f\in\mathcal{C}_{k/(k+1)}$ and hence the following
inequality holds: \begin{equation}\label{kucv-suffi}
\left|\frac{zf''(z)}{f'(z)} \right| < \frac{1}{k+1}. \end{equation}
The inequality \eqref{kucv-suffi} yields
\[ k  \left|\frac{zf''(z)}{f'(z)} \right|< \frac{k}{k+1}=
1-\frac{1}{k+1}< 1- \left|\frac{zf''(z)}{f'(z)} \right| <1+ \RE
\left(\frac{zf''(z)}{f'(z)} \right),\]  and this proves that $f\in
k-\UCV$.
\end{proof}

Since $f\in \mathcal{C}_\alpha$ if and only if $zf'\in
\mathcal{S}^*_\alpha$, Theorem~\ref{th1} (1) immediately yields the
following result. It also follows from Theorem~\ref{suffi}(1) and
the inequality $n-\alpha\leq (2-\alpha)(n-1)$, $n\geq 2$.

\begin{corollary}\label{cor2}
Let $\alpha\in[0,1)$. If $f\in \mathcal{A}$ is given by \eqref{f(z)}
and
\[\sum_{n=2}^{\infty} (n-1)|a_n|\leq \frac{1-\alpha}{2-\alpha},\]
then $f\in  \mathcal{S}^*_\alpha$. Further, the bound
$(1-\alpha)/(2-\alpha)$ cannot be replaced by a larger number.
\end{corollary}

\begin{remark}
Theorem~\ref{T2} for the class of parabolic starlike functions of
order $\rho$ was obtained by Ali \cite[Theorem 3.1, p.\ 564]{ali1}
by using a two variable characterization of a corresponding class of
uniformly convex functions.  However, Theorem~\ref{T2}  follows
directly from the Corollary~\ref{cor2} and his sufficient condition
\cite[Theorem 2.2, p.\ 563]{ali1} for functions to be parabolic
starlike of order $\rho$.
\end{remark}

\begin{theorem} \label{th5}Let $\alpha\in [0,1)$ and $f\in \mathcal{A}$
be given by \eqref{f(z)}.

\begin{enumerate}
  \item [(1)]If the inequality $\sum_{n=2}^{\infty}n|a_n|\leq  1-\alpha$ holds,
then $f\in \mathcal{ S}^*_\alpha$.
  \item[(2)] If  the inequality
$\sum_{n=2}^{\infty}n^2|a_n|\leq1-\alpha$ holds, then $f\in
\mathcal{C}_{\alpha}$.

\item[(3)] If  the inequality
$\sum_{n=2}^{\infty}n^2|a_n|\leq 4(1-\alpha)/(2-\alpha)$ holds, then
$f\in \mathcal{S}^*_{\alpha}$ and the bound $4(1-\alpha)/(2-\alpha)$
is sharp.
\end{enumerate}

\end{theorem}
\begin{proof}The first two parts  follow from  Theorem \ref{suffi} and the simple inequality
$n-\alpha < n $. The third part follows from Theorem \ref{suffi}(1)
and the identity: $ (n-\alpha)\leq n^2(2-\alpha)/4 \quad (n\geq2)$.
 The result is sharp for the function $f_0$ given by
\[f_0(z)=z-\frac{1-\alpha}{2-\alpha}z^2.\qedhere\]\end{proof}

\section{The subclass $\mathcal{R}(\alpha,\beta)$}\noindent
Recall that the class $\mathcal{R}(\alpha, \beta)$  consists of
functions $f$ satisfying the   inequality
\begin{equation}\label{sec3R} \RE
\left(\frac{zf'(z)}{f(z)}\Big(\alpha\frac{zf''(z)}{f'(z)}+1
\Big)\right)>\beta, \quad ( \beta<1,\ \alpha\in
\mathbb{R}).\end{equation} The following Lemma~\ref{lem3.1} provides
a sufficient coefficient condition for functions $f$ to belong to
the class $\mathcal{R}(\alpha,\beta)$. In this section, conditions
are determined so that the sufficient coefficient condition in
Lemma~\ref{lem3.1} implies starlikeness and convexity of certain
order. Also the sufficient coefficient inequalities are considered
for functions to belong to the class $\mathcal{R}(\alpha,\beta)$.

\begin{lemma} {\rm\cite[cf. Theorem 6, p.\ 412]{liu}}\label{lem3.1}  Let
$  \beta<1$, and $\alpha\in \mathbb{R}$.   If $f\in\mathcal{A}$
satisfies the inequality \begin{equation}\label{th3.2e1}
\sum_{n=2}^{\infty}\big(\alpha
n^2+(1-\alpha)n-\beta\big)|a_n|\leq1-\beta,
\end{equation} then
$f\in\mathcal{R}(\alpha, \beta)$.
 \end{lemma}

It should be remarked that Lemma~\ref{lem3.1} reduces to
Theorem~\ref{suffi}\eqref{suffi1a} in the special case $\alpha=0$.
The following theorem provides sufficient coefficient conditions for
functions to belong to either $\mathcal{R}(\alpha, \beta)\cap
\mathcal{S^*}_{\eta}$ or $\mathcal{R}(\alpha, \beta)\cap
\mathcal{C}_{\eta}$.

\begin{theorem} \label{th6}Let $ \beta<1$ and  $\alpha>0$. If the function $f\in\mathcal{A}$
satisfies the inequality \eqref{th3.2e1}, then the following holds.
\begin{enumerate}
\item[(1)] The function $f$ is in the class $ \mathcal{S^*}_{\eta}$ for  $\eta\leq
(2\alpha+\beta)/(2\alpha+1)$ and the bound
$(2\alpha+\beta)/(2\alpha+1)$ is sharp.

\item[(2)] The function $f$ is in the class $\mathcal{C}_{\eta}$ for $\eta\leq
(\alpha-1+\beta)/\alpha$, $\beta>0$.
\end{enumerate}
\end{theorem}

\begin{proof} (1)  If $\eta\leq \eta_0:=(2\alpha+\beta)/(2\alpha+1)$, then
$\mathcal{S}^*_{\eta_0}\subset  \mathcal{S}^*_\eta$.  Hence it is
enough to prove that $f\in \mathcal{S}^*_{\eta_0}$. The inequality
\begin{align*}
&(2\alpha +1)n-2\alpha\leq \alpha n^2+(1-\alpha)n \quad (n\geq2,\
\alpha\geq0)
\end{align*}
together with \eqref{th3.2e1} shows  that
\begin{align*}
\sum_{n=2}^{\infty}(n-\eta_0)|a_n| &=\sum_{n=2}^{\infty}
\frac{(2\alpha+1)n-2\alpha-\beta}{2\alpha+1}|a_n|\\
& \leq\sum_{n=2}^{\infty}\frac{\alpha n^2+(1-\alpha)n-\beta}{2\alpha+1} |a_n| \\
&\leq\frac{1-\beta}{2\alpha+1}= 1-\eta_0. \end{align*}  Thus, by
Theorem~\ref{suffi}(1),  $f\in \mathcal{S^*}_{\eta_0}$. The result
is sharp for the function $f_0\in\mathcal{S^*}\left(\eta_0\right)$
given by
\[f_0(z)=z-\frac{1-\beta}{2\alpha+2-\beta}z^2
.\]

(2) If $\eta\leq \eta_0:=(\alpha-1+\beta)/\alpha$, then
$\mathcal{C}_{\eta_0}\subset  \mathcal{C}_\eta$.  Hence it is enough
to prove that $f\in \mathcal{C}_{\eta_0}$.  The inequality
\[\alpha n^2+(1-\alpha)n-n\beta  \leq \alpha n^2+(1-\alpha)n-\beta
\quad( n\geq2, \quad \beta\geq0),\]  together with \eqref{th3.2e1}
yields
\begin{align*} \sum_{n=2}^{\infty} n(n-\eta_0)|a_n|&
=\frac{1}{\alpha} \sum_{n=2}^{\infty}\big(\alpha n^2+(1-\alpha)n-n\beta\big)|a_n|\\
 & \leq \frac{1}{\alpha}  \sum_{n=2}^{\infty}\left(\alpha
 n^2+(1-\alpha)n-\beta\right)|a_n|\\
 & \leq \frac{1-\beta}{\alpha}=
 1-\eta_0.
 \end{align*}   Thus,   by
Theorem~\ref{suffi}(2), $f\in \mathcal{C}_{\eta_0}$.
\end{proof}

Along the same line as Theorem~\ref{th1},  the following theorem
provides a sufficient coefficient inequality for functions to belong
to the class $\mathcal{R}(\alpha,\beta)$.

\begin{theorem} \label{th7}Let $\beta<1$, $\alpha\in \mathbb{R}$ and $f\in
\mathcal{A}$.
\begin{enumerate}
\item[(1)] If the function $f$ satisfies the inequality $\sum_{n=2}^{\infty}n(n-1)|a_n| \leq
2(1-\beta)/(2\alpha+2-\beta),$ then $f\in\mathcal{R}(\alpha,
\beta)$. The bound $2(1-\beta)/(2\alpha+2-\beta)$ is sharp.

\item[(2)] Let $\alpha\leq1$ and $\eta\in \mathbb{R}$ be defined by \[ \eta=
\begin{cases}
4(1-\beta)/(3\alpha+1),\quad & \alpha+\beta>1,\\
4(1-\beta)/(2\alpha+2-\beta),\quad & \alpha+\beta\leq 1.
\end{cases}\]  If the function $f$ satisfies the inequality  $\sum_{n=2}^{\infty}n^2|a_n|\leq \eta $,
then $f\in\mathcal{R}(\alpha, \beta)$. When $\alpha+\beta\leq 1$,
the result is sharp.
\end{enumerate}
\end{theorem}

\begin{proof}

(1) Let the function $f$ satisfy the inequality
\[\sum_{n=2}^{\infty}n(n-1)|a_n| \leq
2(1-\beta)/(2\alpha+2-\beta).\]
Since,  for $ n\geq2$, \[ 2\alpha n^2+2(1-\alpha)n-2\beta \leq
(2\alpha+2-\beta)n(n-1),\]  it follows that
\[\sum_{n=2}^{\infty}\big(\alpha n^2+(1-\alpha)n-\beta\big) |a_n|
\leq\frac{1}{2}\sum_{n=2}^{\infty}n(n-1)(2\alpha+2-\beta)|a_n|\leq
1-\beta,\] and so, by Lemma~\ref{lem3.1}, $f\in\mathcal{R}(\alpha,
\beta)$. The result is sharp for the  function $f_0\in
\mathcal{R}(\alpha, \beta)$ given by
\[f_0(z)=z-\frac{1-\beta}{2\alpha+2-\beta}z^2.\]

(2) Let  $\alpha+ \beta>1$ and the function $f$ satisfy the
inequality  \[ \sum_{n=2}^{\infty}n^2|a_n|\leq
4(1-\beta)/(3\alpha+1).\] In this case, the use of the inequality
\[  4\left(\alpha n^2+(1-\alpha)n-\beta\right)\leq (3\alpha+1)n^2\quad (n\geq 2)\]
readily  yields
\[\sum_{n=2}^{\infty}\big(\alpha n^2+(1-\alpha)n-\beta\big)|a_n|
\leq\frac{3\alpha+1}{4}\sum_{n=2}^{\infty}n^2|a_n|\leq 1-\beta .\]
Lemma~\ref{lem3.1} then shows that $f\in\mathcal{R}(\alpha, \beta)$.

Now, let   $ \alpha+\beta<1$ and   the function $f$ satisfy
$\sum_{n=2}^{\infty}n^2|a_n|\leq 4(1-\beta)/(2\alpha+2-\beta)$. In
this case, the inequality
\[ 4\left(\alpha
n^2+(1-\alpha)n-\beta\right)\leq n^2(2\alpha+2-\beta) \quad (n\geq
2)
\]  shows that
\[\sum_{n=2}^{\infty}\big(\alpha n^2+(1-\alpha)n-\beta\big)|a_n|
\leq\frac{1}{4}\sum_{n=2}^{\infty}n^2(2\alpha+2-\beta)|a_n|\leq
1-\beta,\] and hence, by Lemma~\ref{lem3.1},
$f\in\mathcal{R}(\alpha, \beta)$. The function
$f_0\in\mathcal{R}(\alpha, \beta)$ given by
\[f_0(z)=z-\frac{1-\beta}{2\alpha+2-\beta}z^2\] shows that  the result is  sharp.
\end{proof}

\section{Functions with negative coefficients}\noindent
In this section, certain classes of functions with negative
coefficients are investigated. The class of functions with negative
coefficients, denoted by $\mathcal{T}$,  consists of the functions
$f$ of the form
\begin{equation}\label{sec3eq1}
f(z)=z-\sum_{n=2}^{\infty}a_nz^n\quad ( a_n\geq 0).
\end{equation}
Denote by $\mathcal{TS}^*(\alpha)$, $\mathcal{TS}^*_\alpha$ and
$\mathcal{TC}(\alpha)$, and  $\mathcal{TC}_\alpha$ the respective
subclasses of functions with negative coefficients in
$\mathcal{S}^*(\alpha)$, $\mathcal{S}^*_\alpha$ and
$\mathcal{C}_\alpha$ and $\mathcal{C}(\alpha)$. For  starlike and
convex functions functions with negative coefficients,  Silverman
\cite{silver} proved the following theorem.

\begin{theorem}\label{A} Let $\alpha\in[0,1)$. If $f\in \mathcal{T}$ is given by
\eqref{sec3eq1}, then \[ f\in \mathcal{TS}^*(\alpha)
\Longleftrightarrow f\in \mathcal{TS}^*_\alpha \Longleftrightarrow
\sum_{n=2}^{\infty}(n-\alpha)a_n\leq 1-\alpha,\] and
\[ f\in \mathcal{TC}(\alpha)\Longleftrightarrow f\in \mathcal{TC}_\alpha
\Longleftrightarrow \sum_{n=2}^{\infty}n(n-\alpha)a_n\leq
1-\alpha.\]
\end{theorem}
 For  functions with
negative coefficients, the next theorem  proves the equivalence of
the inequalities $\sum_{n=2}^{\infty}n(n-1)a_n\leq \beta$ and
$|f''(z)|<\beta$.

\begin{theorem}Let $\beta>0$. If the function  $f\in \mathcal{T}$ is given by
\eqref{sec3eq1}, then \[ |f''(z)|\leq\beta\Longleftrightarrow
\sum_{n=2}^{\infty}n(n-1)a_n\leq \beta.\]
\end{theorem}

\begin{proof}If $f$ satisfies the coefficient inequality  $\sum_{n=2}^{\infty}n(n-1)a_n\leq \beta$, then
\[ |f''(z)|\leq \sum_{n=2}^\infty n(n-1) a_n  |z|^{n-2}
\leq \sum_{n=2}^\infty n(n-1) a_n   \leq\beta .\]

The converse follows, by allowing $z\rightarrow1^-$, in
\[|f''(z)|=\left|\sum_{n=2}^{\infty}n(n-1)a_nz^{n-2}\right|\leq\beta.\qedhere\]
\end{proof}

\begin{remark}It is well known that if the function $f\in\mathcal{A}$ satisfies the inequality
$|f''(z)|\leq \beta $ for $0<\beta\leq1$, then $f\in \mathcal{S}^*$
and if $|f''(z)|\leq \beta $ for $0<\beta\leq1/2$, then $f\in
\mathcal{C}$ \cite[Theorem 1, p.1861]{vsingh}. \end{remark}

\begin{theorem} If the function $f\in\mathcal{TC}(\alpha)$, $0\leq \alpha<1$,
then the following holds:
\begin{enumerate}
\item[(1)] The inequality $\sum_{n=2}^{\infty}n
a_n\leq(1-\alpha)/(2-\alpha)$ holds and the bound
$(1-\alpha)/(2-\alpha)$ is sharp.

\item[(2)]  The inequality $\sum_{n=2}^{\infty}n (n-1)a_n\leq1-\alpha$ holds.

\item[(3)]  The inequality $\sum_{n=2}^{\infty} (n-1)a_n\leq(1-\alpha)/2(2-\alpha)$ holds and
the bound $(1-\alpha)/2(2-\alpha)$ is sharp.

\item[(4)] The inequality $\sum_{n=2}^{\infty}n^2a_n\leq 2(1-\alpha)/(2-\alpha)$ holds and
the bound $2(1-\alpha)/(2-\alpha)$ is sharp.
\end{enumerate}

\end{theorem}
\begin{proof}The results follow respectively from Theorem~\ref{A} and the simple inequalities
$2-\alpha\leq n-\alpha$, $n-1\leq n-\alpha$, $2(2-\alpha)(n-1)\leq
n(n-\alpha)$, and $n^2(2-\alpha) \leq 2 n(n-\alpha)$ satisfied for
$n\geq 2$. The  sharpness follows by considering the function $f_0$
given by
\[f_0(z)=z-\frac{1}{2}\frac{1-\alpha}{2-\alpha}z^2.\qedhere\]
\end{proof}

Alexander theorem  between $\TC(\alpha)$ and $\TS(\alpha)$
immediately yields the following corollary.

\begin{corollary}If the function $f\in\mathcal{TS}^*(\alpha)$, $0\leq \alpha<1$,
then the following holds.
\begin{enumerate}
\item[(1)] The inequality $\sum_{n=2}^{\infty}  a_n\leq(1-\alpha)/(2-\alpha)$ holds and the bound
$(1-\alpha)/(2-\alpha)$ is sharp.

\item[(2)] The inequality $\sum_{n=2}^{\infty} (n-1)a_n\leq
1-\alpha$ holds.

\item[(3)]  The inequality $\sum_{n=2}^{\infty}n
a_n\leq2(1-\alpha)/(2-\alpha)$ holds and the bound
$2(1-\alpha)/(2-\alpha)$ is sharp.
\end{enumerate}
\end{corollary}

In the remaining part of this section, the properties of functions
with negative coefficients  belonging to the class
$\mathcal{R}(\alpha, \beta)$ are investigated. The class of all
functions  with negative coefficients  belonging to the class
$\mathcal{R}(\alpha, \beta)$ is denoted in the sequel by
$\mathcal{TR}(\alpha, \beta)$. To begin with, the following lemma is
needed.

\begin{lemma}\cite[Theorem 8, p.414]{liu}\label{LemaTR} If  $\beta<1$, $\alpha\in \mathbb{R}$ . If
 $f\in \mathcal{T}$, then
\[ f\in\mathcal{TR}(\alpha, \beta) \Longleftrightarrow \sum_{n=2}^{\infty}\big(\alpha
n^2+(1-\alpha)n-\beta\big)a_n\leq1-\beta .
\]
\end{lemma}

\begin{corollary} \label{C6} If $f\in\mathcal{TR}(\alpha, \beta)$ with $\beta<1$, $\alpha>0$, then
the following holds:
\begin{enumerate}
\item[(1)]
The function $f\in \mathcal{TS^*}_{\eta}$ for  $\eta\leq
(2\alpha+\beta)/(2\alpha+1)$ and the bound
$(2\alpha+\beta)/(2\alpha+1)$ is sharp.

\item[(2)] The function  $f\in \mathcal{TC}_{\eta}$ for $\eta\leq (\alpha-1+\beta)/\alpha$.
\end{enumerate}
\end{corollary}

\begin{proof} The result follows from Lemma \ref{LemaTR} and Theorem \ref{th6}.\end{proof}

 The next result shows that
$\mathcal{TC}\big( (2\alpha+3\beta-2)/(2\alpha+\beta) \big) \subset
\mathcal{TR}(\alpha,\beta)$ for $0\leq \beta<1$, $\alpha\in
\mathbb{R}$.

\begin{theorem}Let $0\leq \beta<1$, and $\alpha>0$. If \,$\eta\geq
(2\alpha+3\beta-2)/(2\alpha+\beta)$, then
$\mathcal{TC}(\eta)\subseteq  \mathcal{TR}(\alpha,\beta)$.
\end{theorem}

\begin{proof}For $\eta_0\leq \eta$,
$\mathcal{TC}(\eta)\subset\mathcal{TC}(\eta_0)$ and therefore it is
enough to prove $\mathcal{TC}(\eta_0)\subseteq
\mathcal{TR}(\alpha,\beta)$ where
$\eta_0=(2\alpha+3\beta-2)/(2\alpha+\beta)$.  For $n\geq2$, the
inequality \[2\alpha n^2+2(1-\alpha)n-2\beta  \leq
n\big((2\alpha+\beta)n-(2\alpha+3\beta-2)\big) \] holds and
therefore
\begin{align*}
\sum_{n=2}^{\infty}\big(\alpha n^2+(1-\alpha)n-\beta\big) a_n
&\leq\frac{1}{2}\sum_{n=2}^{\infty}
n\big((2\alpha+\beta)n-(2\alpha+3\beta-2)\big)a_n\\
& = \frac{2\alpha+\beta}{2}\sum_{n=2}^\infty n(n-\eta_0)a_n\\
& \leq \frac{2\alpha+\beta}{2}(1-\eta_0)\\
&= 1-\beta,\end{align*}  and, by Lemma~\ref{LemaTR},  $f\in
\mathcal{TR}(\alpha,\beta)$.
\end{proof}
\begin{theorem} Let $\beta<1$, and $\alpha\in \mathbb{R}$. If $f\in\mathcal{TR}(\alpha, \beta)$,
then
\begin{enumerate}
  \item[(1)] $\sum_{n=2}^{\infty}n(n-1)a_n\leq(1-\beta)/\alpha$ when $\alpha>0$.

  \item[(2)]  $\sum_{n=2}^{\infty}(n-1)a_n\leq \eta$ where \[\eta=
  \begin{cases}(1-\beta)/(1-\alpha),\quad &\beta<3\alpha+1,\quad 0\leq\alpha<1\\
(1-\beta)/(2\alpha+2-\beta),\quad &\beta\geq 3\alpha+1, \quad
0\leq\alpha.
\end{cases}\]
The result for $\beta>3\alpha+1$ is sharp.

  \item[(3)] For $0\leq
\alpha\leq 1$, $\sum_{n=2}^{\infty}n^2a_n\leq \eta$ where\[\eta=
  \begin{cases}(1-\beta)/\alpha,\quad &\beta<2(1-\alpha),\alpha>0\\
4(1-\beta)/(2\alpha+2-\beta),\quad &\beta\geq
2(1-\alpha),,\beta\geq0.
  \end{cases}\]
   The result for $\beta>2(1-\alpha)$ is sharp.

  \item[(4)] $\sum_{n=2}^{\infty}na_n\leq
   2(1-\beta)/(2\alpha+2-\beta)$, $\alpha, \beta \geq 0$.
   The result is sharp.
\end{enumerate}
\end{theorem}
\begin{proof} Since  $f\in\mathcal{TR}(\alpha, \beta)$,  by Lemma~\ref{LemaTR}, the
following inequality holds:
\[ \sum_{n=2}^{\infty}\big(\alpha
n^2+(1-\alpha)n-\beta\big)a_n\leq1-\beta .\] This inequality is used
throughout the proof of this theorem.

(1) Since  \[\alpha n(n-1)\leq\alpha n^2+(1-\alpha)n-\beta\quad
n\geq2,\] it readily follows that
\[\sum_{n=2}^{\infty}  n( n-1) a_n
\leq\sum_{n=2}^{\infty}\frac{\alpha n^2+(1-\alpha)n-\beta}{\alpha
}a_n\leq \frac{1-\beta}{\alpha}.\]

(2) If  $\beta<3\alpha+1$, then, for $n\geq 2$,
\[(n-1)(1-\alpha)\leq \alpha n(n-1) +n-\beta \quad (n\geq 2)\]
and an use of  this inequality shows  that
\[\sum_{n=2}^{\infty} ( n-1) a_n
\leq\sum_{n=2}^{\infty}\frac{\alpha
n^2+(1-\alpha)n-\beta}{1-\alpha}a_n\leq \frac{1-\beta}{1-\alpha}.\]

If  $\beta>3\alpha+1$, then the inequality
\begin{align*}
&(n-1)(2\alpha+2-\beta)\leq \alpha n^2+n(1-\alpha)-\beta\quad (n\geq
2)
\end{align*}
 shows that
\[\sum_{n=2}^{\infty} (n-1)a_n
\leq\sum_{n=2}^{\infty}\frac{\alpha
n^2+(1-\alpha)n-\beta}{2\alpha+2-\beta}a_n\leq
\frac{1-\beta}{2\alpha+2-\beta}.\]

(3) If $\beta<2(1-\alpha)$, the result follows from  inequality
\begin{align*}
\alpha n^2&\leq\alpha n^2+2(1-\alpha)-\beta\leq\alpha
n^2+n(1-\alpha)-\beta.
\end{align*}
Using this inequality, it follows that
\[\sum_{n=2}^{\infty} n^2 a_n\leq\sum_{n=2}^{\infty}
\frac{\alpha n^2+(1-\alpha)n-\beta}{\alpha} a_n\leq
\frac{1-\beta}{\alpha}.\]

In the case $\beta\geq2(1-\alpha)$,  the inequality\[
 n^2(2\alpha+2-\beta)\leq 4(\alpha n^2+(1-\alpha)n-\beta) \quad
 (n\geq 2)\] shows that
\[\sum_{n=2}^{\infty}n^2a_n
\leq\sum_{n=2}^{\infty}\frac{4(\alpha
n^2+(1-\alpha)n-\beta)}{2\alpha+2-\beta}a_n\leq\frac{4(1-\beta)}{2\alpha+2-\beta}.\]

(4) For $\alpha,\beta>0$, the inequality
\[ (2\alpha+2-\beta)n\leq 2\big(\alpha
n^2+(1-\alpha)n-\beta\big)\] shows that
\[\sum_{n=2}^{\infty}na_n\leq
  \sum_{n=2}^{\infty}\frac{2\big(\alpha n^2+(1-\alpha)n-\beta\big)}{2\alpha+2-\beta}a_n
  \leq\frac{2(1-\beta)}{2\alpha+2-\beta}.\]

The  sharpness can be seen by  considering  the function $f_0$ given
by
   \[f(z)=z-\frac{1-\beta}{2\alpha+2-\beta}z^2\in\mathcal{TR}(\alpha,\beta).\qedhere\]
\end{proof}


\section{Applications to Gaussian  hypergeometric functions}
\noindent For  $a,b,c\in \mathbb{C}$  with $c\neq0,-1,-2,\dotsc$,
the \emph{Gaussian hypergeometric function} is defined by
\[F(a,b;c;z):=\sum_{n=0}^{\infty}\frac{(a)_n(b)_n}{(c)_n(1)_n}z^n=1+\frac{ab}{c}\frac{z}{1!}
+\frac{a(a+1)b(b+1)}{c(c+1)}\frac{z^2}{2!}+\cdots,\] where
$(\lambda)_n$ is Pochhammer symbol defined, in terms of the Gamma
function, by
\[(\lambda)_n=\frac{\Gamma(\lambda+n)}{\Gamma(\lambda)}\quad (
n=0,1,2,\dotsc).\]The series converges absolutely on $\mathbb{D}$.
It also converges  on $|z|=1$ when $\RE(c-a-b)>0$. For
$\RE(c-a-b)>0$, the value of the hypergeometric function
$F(a,b;c;z)$ at $z=1$ is related to Gamma function by the following
Gauss summation formula
\begin{equation}\label{Fat1}F(a,b;c;1)
=\frac{\Gamma(c)\Gamma(c-a-b)}{\Gamma(c-a)\Gamma(c-b)}\quad
(c\neq0,-1,-2,\ldots).\end{equation}

By making use of Theorem~\ref{suffi}, Silverman \cite{silver2}
determined conditions on $a,b,c$ so that  the function $zF(a,b;c;z)$
belongs to certain subclasses of the starlike  and convex functions.
In the following theorem, the conditions on the parameters $a,b,c$
are determined so that the function $zF(a,b;c;z)$ belongs to the
class $\mathcal{R}(\alpha,\beta)$. For other classes investigated in
this paper, similar results are true but the details are omitted
here. The proof  follows directly by the corresponding theorems in
the previous sections, the Gauss summation formula for the Gaussian
hypergeometric functions and some algebraic manipulation; the
details of the proofs are omitted as they are similar to those of
Silverman \cite{silver2} and Kim and Ponnusamy \cite{kim}.

\begin{theorem}
Let  $a,b\in\mathbb{C}$ and $c\in\mathbb{R}$ satisfy  either
\[ F(|a|,|b|;c;1)\left(\frac{(|a|)_2(|b|)_2}{(c-|a|-|b|-2)_2}+\frac{2|ab|}{c-|a|-|b|-1}\right)
\leq\frac{2(1-\beta)}{2\alpha+2-\beta},\] for $ c>|a|+|b|+2$,
$\alpha\geq0$, $\beta<1$, or
\[ F(|a|,|b|;c;1)\left(\frac{(|a|)_2(|b|)_2}{(c-|a|-|b|-2)_2}+\frac{3|ab|}{c-|a|-|b|-1}+1\right)
\leq\frac{6-5\beta+2\alpha}{2\alpha+2-\beta},\] for $c>|a|+|b|+2$,
$1-\alpha\geq\beta$, $\alpha\in[0,1)$, then  the function
$zF(a,b;c;z)\in\mathcal{R}(\alpha, \beta)$. In the case
$b=\overline{a}$, the range of $c$ in either case can be improved to
$c>\max\{0,2(1+\RE a)\}$.
\end{theorem}

\end{document}